\documentclass[12pt, a4paper]{amsart}
\usepackage[margin=1.2in]{geometry}
\usepackage{amsmath}
\usepackage{amsthm}
\usepackage{color}
\usepackage{graphicx}
\usepackage{float}
\usepackage{young}
\usepackage[vcentermath]{youngtab}
\usepackage{amsfonts}
\usepackage{hhline}
\usepackage{amssymb}
\usepackage{amscd}
\input xy
\xyoption{all}
\usepackage{booktabs}
\usepackage{array}
\usepackage{multirow}
\numberwithin{equation}{section}
\newtheorem{theorem}{Theorem}[section]

\newtheorem{cor}[theorem]{Corollary}
\newtheorem{lemma}[theorem]{Lemma}
\newtheorem{prop}[theorem]{Proposition}

\newtheorem{defi}{Definition}[section]

\theoremstyle{definition}

\newtheorem{rem}{Remark}[section]

\renewcommand\j{{\bf j}}

\newcommand\ov{\overline}

\newcommand\be{\beta}

\newcommand\g{\mathfrak g}

\newcommand\h{\mathfrak h}

\newcommand\n{\mathfrak n}
\newcommand\bb{\mathfrak b}
\newcommand\D{\Delta}
\newcommand\Da{\widehat\Delta}
\newcommand\Dp{\Delta^+}
\newcommand\Pia{{\widehat\Pi}}
\newcommand\Dap{\widehat\Delta^+}
\newcommand\Wa{\widehat{W}}
\renewcommand\d{\delta}

\renewcommand\a{\alpha}

\renewcommand\b{\bb}

\renewcommand\i{{\mathfrak  i}}

\newcommand\nat{\mathbb N}
\newcommand\ganz{\mathbb Z}
\renewcommand\j{\mathfrak j}
\newcommand\s{\sigma}

\newcommand\e{\epsilon}
\newcommand\C{\mathbb C}
\newcommand\R{\mathbb R}

\newcommand{\Wab}{\mathcal W}
\newcommand{\Waab}{\mathcal W}

\renewcommand{\l}{\lambda}

\newcommand{\Ab}{\mathfrak{Ab}}
\begin{document}
\title{Symmetries of the poset of abelian ideals in a  Borel subalgebra}
\author[Cellini]{Paola Cellini}
\author[M\"oseneder Frajria]{Pierluigi M\"oseneder Frajria}
\author[Papi]{Paolo Papi}
\begin{abstract} Elaborating on Suter's paper \cite{S}, we provide a detailed description of  the automorphism group of the poset of abelian ideals in a Borel subalgebra of a finite dimensional complex simple Lie algebra.
\end{abstract}
\maketitle
\par\noindent
\section{Introduction}
This paper stems out from the attempt  to get a better understanding of  the final part of Suter's paper \cite{S}, where   the symmetries of the Hasse graph of the poset $\Ab$ of abelian ideals of a Borel subalgebra  in a finite dimensional complex simple Lie algebra $\g$ are analyzed.  After Kostant's seminal paper
\cite{KostIMRN}, abelian ideals of Borel subalgebras have been intensively studied. The theory of abelian ideals of Borel subalgebras
offers a wide variety of applications, ranging from representation theory of Kac-Moody algebras to combinatorics and number theory.  But its distinctive feature  is  to  provide a framework linking the theory of affine Weyl groups, the structure theory for the exterior algebra
of $\g$ as a $\g$-module, and many combinatorial aspects of the representation theory of $\g$. A glimpse on these connections is given in Section 3, where we also provide a concise description of the many ways known in literature to encode abelian ideals of  Borel subalgebras. \par
The  main result of  the present paper is a rigidity statement  about the poset structure of $\Ab$; we give a detailed proof  that its symmetries are exactly  the ones induced by automorphisms  of the Dynkin diagram, with just one exception in type $C_3$. See Theorem \ref{T}. 
Our goal was to find proofs which were as far as possible independent from the inspection of the global structure of the poset: the outcome of our efforts is that proofs only require either global inspections 
in rank at most $4$ or ``local'' inspections, which may be easily performed using Bourbaki's Tables.\par
The non trivial part in the proof of  Theorem \ref{T} consists in showing that an automorphism of the abstract poset $\Ab$ is indeed induced by an automorphism of the Dynkin diagram. The proof of this fact is discussed in Section 5. The main theorem is also used in Section 6 to discuss the symmetries of the Hasse graph of $\Ab$,
which is the original result by Suter.
In Section 4 we take the opportunity of discussing in detail some folklore results relating the automorphisms of the Dynkin diagram, the automorphisms of the extended Dynkin diagram, and the center of the connected simply connected Lie  group corresponding to $\g$. We also recover in our setting  the dihedral symmetry of a remarkable subposet  in the  Young lattice discovered by Suter in   \cite{Su}
and further discussed in \cite{S}; it is worthwhile to note that this symmetry recently got a renewed interest in literature: see \cite{1}, \cite{2}, \cite{Su3}.\par
{\it Acknowledgement.}
We would like to thank D. Panyushev for drawing our attention on the paper \cite{Su} and R. Suter for correspondence. We also would like to thank the anonymous referee for his/her careful reading of the paper.

\section{Setup}
Let 
 $\g$ be a finite-dimensional complex simple Lie algebra. Let  $\h\subset\g$ be a Cartan subalgebra. Denote by  $\D$ the corresponding irreducible root system, and by $W$ the  Weyl group of $\D$. 
 Fix a positive system $\Dp$ and let  $\Pi=\{\a_1,\ldots,\a_n\}$ be the corresponding basis of simple roots. Recall that there are at most two possible length for roots, which are correspondingly termed as long and short. We stipulate that
 roots are long if just one length occurs. 
 Let $(\cdot,\cdot)$ be {\it half} the Killing form of $\g$. For $\a\in\D\subset\h^*$, denote by $\a^\vee\in\h$ the corresponding coroot. Let  
 $\{\omega_1,\ldots,\omega_n\}\subset\h^*,\,\{\varpi_1,\ldots,\varpi_n\}\subset \h$ denote the dual bases of $\Pi^\vee, \Pi$, respectively, and set $\rho=\sum_{i=1}^n\omega_i$. 
 Let $P, Q$ be the  weight and root lattices, viewed as  posets via $\beta\leq\gamma\iff \gamma-\beta\in\ganz_+\Pi$. We also denote by $P^\vee, Q^\vee$ the coweight and coroot lattices.\par
  \par
 Let $F$ be the space of affine-linear functions on $V=\R\otimes_\ganz Q^\vee$. Endow $F$ with a symmetric bilinear form induced by $(\cdot,\cdot)$  on the linear part and extended by zero on the affine part.
 
  For $\a\in \D,j\in \ganz$, consider the following element of $F$:
 $$a_{\a,j}(v)=\a(v)+j.$$
It is shown in \cite{Mac} that the set  $\Da=\{a_{\a,j}\mid \a\in\D,j\in\ganz\}$ is an affine root system in $F$. For $\a\in \D,\ j\in \ganz$ let $s_{\a,j}$  be the affine reflection around the hyperplane $\a(x)=j$. Explicitly, $s_{\a,j}(v)=v-a_{\a,-j}(v)\a^\vee$. Let $\Wa$ be the subgroup of $Isom(V)$ generated by $\{s_{\a,j}\mid a_{\a,j}\in\Da\}$.
 Let $t_v$ be the translation by $v$.  It is well-known that $\Wa=W\ltimes Q^\vee$ (where $Q^\vee$ is viewed inside $\Wa$ via $\a^\vee\mapsto t_{\a^\vee}$) and that it is  a Coxeter group with generating set $s_0=s_{\theta,1}=t_{\theta^\vee}s_{\theta,0},s_i=s_{\a_i,0},\,i=1,\ldots,n$. Here $\theta=\sum_{i=1}^nm_i\a_i$ is the highest root of $\D$.
 
 A fundamental domain for the action of $\Wa$ on $V$ is given by 
 $$\{v\in V\mid \a(v)\geq 0\,\forall\,\a\in\Dp,\,\theta(v)\leq 1\}.$$
 Identifying $V$ and $V^*$ by means of $(\cdot,\cdot)$, we
can also define an action of $\Wa$ on $V^*$; then
 $$C_1=\{\lambda\in V^*\mid (\a,\lambda)\geq 0\,\forall\,\a\in\Dp,\,(\theta,\lambda)\leq 1\}$$
 is a fundamental domain for this action, called the {\it fundamental alcove}. We will refer to the alcoves as the $\Wa$-translates of $C_1$.  \par

The set 
 $$\Dap=\{a_{\a,j}\mid \a\in\D,j>0\}\cup\{a_{\a,0}\mid \a\in\Dp\}$$
 can be shown to be a set of positive roots in $\Da$ and the corresponding set of simple roots is 
 $\Pia=\{\a_0,\ldots,\a_n\}$, where $\a_0=a_{-\theta,1}$ and we identify $\a_i$ with $a_{\a_i,0},\,i=1,\ldots,n$. 
 
 Note that $\Wa$ acts on $F$ (as functions on $V$) and this action preserves $\Da$ and fixes $\d$, the constant function $1$.  Note that if we set $m_0=1$ we have 
 \begin{equation}\label{d}\d=\sum_{i=0}^nm_i\a_i.\end{equation}

  Let $\widehat  A=(a_{ij})_{i,j=0}^n,\,\,a_{ij}=\frac{2(\a_i,\a_j)}{(\a_j,\a_j)}$ be the Cartan matrix associated to $\Pia$, so that $ A=(a_{ij})_{i,j=1}^n$ is the   Cartan matrix of $\g$ (w.r.t. $\Pi$).  

\section{Abelian ideals of Borel subalgebras} Let $\b$ the Borel subalgebra corresponding to our choice of $\h$ and $\Dp$. 
Let us denote by $\Ab$ the poset of abelian ideals of $\b$. We now sum up all the  encodings of $\Ab$ we shall use in the following. For $w\in\Wa$, we set 
$N(w)=\{\a\in\Dap\mid w^{-1}(\a)\in -\Dap\}$. Recall that if $w=s_{i_1}\cdots s_{i_k}$ is a reduced expression of $w$, then 
$$N(w)=\{\a_{i_1},s_{i_1}(\a_{i_2}),\ldots,s_{i_1}\cdots s_{i_{k-1}}(\a_{i_k})\}.$$\par
Also  recall that, in any poset $\mathcal P$, a subset $I\subset \mathcal P$ is an 
order ideal (resp. dual order ideal) if $x\in I,\, y\leq x\implies y\in I$ (resp. $x\in I,\, y\geq x\implies y\in I$). Finally, an antichain $A\subset \mathcal P $ is a subset consisting of mutually non-comparable elements.
\par Spcialize now to the root poset $(\Dp,\leq)$. We say that a dual order ideal $I$ is {\it abelian} if $\a,\be\in I \implies \a+\be\notin \D$.
Let us set the following definitions.
\begin{defi}\
\begin{enumerate}
\item Set 
$$\Wab=\{w\in\Wa\mid  N(w)=\d-\Phi, \Phi\text{  abelian  dual order ideal in $\Dp$}\}.$$
These elements are called {\it minuscule}.
\item The $\rho$-points
 in $2C_1$ are the set of regular elements in $P\cap 2C_1$.
 \item  The {\it weight}  of $\i\in\Ab$ is $\langle\i\rangle =\sum_{\g_\a\subset \i}$.
\end{enumerate}
\end{defi}

\begin{prop}\label{CcPp} The following sets are in bijection with $\Ab$:
\begin{enumerate}
\item the set of abelian dual order ideals in $\Dp$;
\item the set $\Wab$ of minuscule elements in $\Wa$;
\item the set of alcoves in $2C_1$;
\item the set of $\rho$-points
 in $2C_1$;
\item the set of weights of abelian ideals.
\end{enumerate}
\end{prop}
\begin{proof}
If $\i\in\Ab$, then 
$\i=\bigoplus\limits_{\a\in\Phi_\i}\g_\a$, where $\Phi_\i=\{\a\in\Dp\mid\g_\a\subset\i\}$. The fact that $\i$ is an abelian ideal of $\b$ clearly translates into the fact that $\Phi_\i$ is a dual order ideal in $\Dp$ which 
is also abelian.
\par
The set $\Ab$ is related to $\Wa$ by the following idea of Dale Peterson:  if $\i\in\Ab$, the set $\d-\Phi_\i\subset\Dap$ is biconvex, hence there exists 
a unique element $w_\i\in\Wa$ such that $N(w_\i)=\d-\Phi_\i$. 

 An explicit description of the set
$\Waab$ of minuscule elements has been found in
\cite{CP}, where it has been shown that  the alcoves 
\begin{equation}\label{alc}C_\i:=w_\i(C_1)
\end{equation}
 cover $2C_1$.

Recall that we are taking as invariant form on $\h$ half the Killing form so $(\cdot,\cdot)$ is twice the Killing form on $\h^*$.  Lemma 2.2 of \cite{CMP} Ênow shows that $P\cap C_1=\{\rho\}$. Hence in any alcove $C_\i$ there is just one regular element of $P$, which is indeed $w_\i(\rho)$ (hence our terminology).

The fact that the map $\i\mapsto \langle\i\rangle$ is injective has been shown in \cite[Theorem 7]{Kostant}.
\end{proof}
\begin{rem} We single out  two more encodings of $\Ab$ available in  literature:
\begin{enumerate}
\item[(6)] the set $\{\eta\in Q^\vee\mid \eta(\a)\in\{-2-1,0,1\}\ \forall\,\a\in\Dp\}$;
\item[(7)] the set of antichains $A\in\Dp$ such that for any $\a,\beta\in A$ we have $\a+\beta\not\leq\theta$.
\end{enumerate}
To get the first encoding, recall the semidirect product decomposition $\Wa=Q^\vee\rtimes W$ and write $w_\i=t_{\tau_\i}v_\i$ accordingly.
Then the map $\Waab\to \{\eta\in Q^\vee\mid \eta(\a)\in\{-2-1,0,1\}\,\forall\,\a\in\Dp\},\,w_\i\mapsto v_\i^{-1}(\tau_\i)$ is a bijection (see \cite{CP}, \cite{KostIMRN}).\par
For the final statement, recall that any (dual)  order ideal in a finite poset is determined by an antichain. It follows from  \cite[Theorem 1]{CDR} that the antichains giving rise to  abelian ideals are characterized by the property stated in $(7)$. 
\end{rem}

\begin{rem} The term {\it weight} Êused in item $(5)$ has indeed a representation-theoretical meaning: one of the main results of 
\cite{Kostant} is the analysis of the structure of $\bigwedge \g$ as a $\g$-module.  Any commutative subalgebra $\mathfrak a =\oplus_{i=1}^k \C v_i\subseteq\g$ gives rise to a decomposable vector $v_{\mathfrak a}=v_1\wedge\cdots\wedge v_k\in \bigwedge^k  \g$. Let $\mathcal A$ be  the span of all the  vectors $v_{\mathfrak a}$.
A key step  in understanding the structure of $\bigwedge \g$ consists in determining the $\g$-module structure of $\mathcal A$. It turns out that 
$\mathcal A$ is multiplicity free and its highest weight vectors are precisely the $v_{\mathfrak i}$, when $\i$ ranges over $\Ab$. Note that the weight of
  $v_{\mathfrak i}$ is $\langle \i\rangle$. 
  \end{rem}
  
  The next proposition, which is known (see \cite{Koinv} or \cite{S}), shows once more that the map $\i\mapsto \langle\i\rangle$ is an encoding of $\Ab$. For the reader's convenience we reprove it in our setting.

\begin{prop}\label{04} $w_\i(\rho)=\rho+\langle \i\rangle$.
\end{prop}\begin{proof} Recall that the action of $\Wa$ on $V^*$ is obtained by identifying $V$ and $V^*$ using the invariant form $(\cdot,\cdot)$.  Explicitly we have that
$$
s_{0}(\l)=\l-(\l(\theta^\vee)-h^\vee)\theta,\ s_{i}(\l)=\l-\l(\a_i^\vee)\a_i, \ (i>0),
$$
where $h^\vee=\frac{2}{(\theta,\theta)}$ is the dual Coxeter number of $\g$.
In particular, we have that
\begin{equation}\label{srho}
s_{0}(\rho)=\rho+\theta,\ s_{i}(\rho)=\rho-\a_i, \ (i>0).
\end{equation}

We  identify $V^*$ and $F/\C\d$; let $\l\mapsto \bar \l$  be the projection map from $F$ to $V^*$. Since $\d$ is fixed by $\Wa$, then  $\l\mapsto \ov{w\l}$ defines a (linear) action of $\Wa$ on $V^*$. Note that
$$w(\l)-w(\mu)=\ov{w(\l-\mu)}.
$$
We now prove by induction on $\ell(w)$ that 
\begin{equation}\label{wrhominusrho}
\rho-w(\rho)=\ov{\langle N(w)\rangle}.
\end{equation}
Indeed, if $\ell(w)=0$ there is nothing to prove, while, if $w=vs_i$ with $\ell(w)=\ell(v)+1$, then 
$$\rho-w(\rho)=\rho-v(\rho)+\ov{v(\rho-s_i(\rho))}=\ov{\langle N(v)\rangle}+\ov{v(\a_i)}=\ov{\langle N(w)\rangle}.
$$
Now observe that, by \eqref{wrhominusrho}, $w_\i(\rho)-\rho=
-\ov{\langle N(w_\i)\rangle}=\langle \i\rangle$.
\end{proof} 

 Recall that $W$ can be endowed with the following partial order, called the {\it left weak Bruhat order}: for $w,w'\in W$ we define $w\leq w'$ if $w'=ws_{i_1}\cdots s_{i_k},\,\ell(ws_{i_1}\cdots s_{i_j})=\ell(w)+j,\,j=1,\ldots,k$.
\begin{rem}\label{posetiso}
Note that $\Ab$ is a poset under inclusion, $\Wab$ is a poset under left weak Bruhat order, and the weights  of ideals and the $\rho$-points have a  poset structure induced by that of $P$.
The maps $\i\mapsto w_\i, \i\mapsto \langle\i\rangle,\,\i\mapsto w_\i(\rho)$ preserve the order. 
Indeed, $u\leq v$ in the left weak Bruhat order if and only if $N(u)\subseteq N(w)$. Moreover if $\i\subseteq \j$,  then $\Phi_\i\subseteq \Phi_\j$ and
$\langle \i\rangle=\sum\limits_{\a\in\Phi_\i}\a\leq \sum\limits_{\a\in\Phi_\j}\a=\langle \j\rangle$. Finally  Proposition \ref{04} guarantees that $\i\mapsto w_\i(\rho)$ is order compatible.
\end{rem}


\section{$Aut(\Pi),\,Aut(\Pia)$ and dihedral symmetries in the Young lattice}
Set
\begin{align*}
Aut(\Pi)&=\{\s:\Pi\leftrightarrow\Pi\mid a_{ij}=a_{\s(i)\s(j)}\},\\
Aut(\Pia)&=\{\s:\Pia\leftrightarrow\Pia\mid a_{ij}=a_{\s(i)\s(j)}\}.
\end{align*}
We are identifying  the action on indices with the action on simple roots.\par
Recall that we set $\theta=\sum_{i=1}^nm_i\a_i$. Define $m_i^\vee$ by setting  $\theta^\vee=\sum_{i=1}^nm^\vee_i\a^\vee_i$.
Also set $m_0=m_0^\vee=1$.
\par\noindent
\begin{lemma}\label{ai} \
\begin{enumerate}
\item If $\s\in Aut(\Pi)$, then $m_i=m_{\s(i)}$ and $m^\vee_i=m^\vee_{\s(i)}$ for all $i=1,\ldots,n$.
\item  If $\s\in Aut(\Pia)$, then $m_i=m_{\s(i)}$ and $m^\vee_i=m^\vee_{\s(i)}$  for all $i=0,\ldots,n$.
\end{enumerate}
\end{lemma}
\begin{proof} (1). Extend $\s$ to an automorphism of $\g$. It preserves the Killing form, hence any bilinear invariant form, since $\g$ is simple. Moreover, it induces an order preserving map on roots, hence fixes $\theta$. In turn $m_i=m_{\s(i)}$ for $i=1,\ldots,n$. Since $m^\vee_i=\frac{(\a_i,\a_i)}{(\theta,\theta)}m_i$, we have  $m^\vee_i=m^\vee_{\s(i)}$ as well.\par
(2). Extend $\s$ by linearity to $F$. Identify $\hat A$ with the operator on $F$ whose matrix  in the basis $\Pia$ is  $\hat A$. Recall that $Ker\,\hat A$ is 1-dimensional generated by $\d$. Then we have 
$$0=\s(\hat A\d)=(\s\circ \hat A \circ \s^{-1})(\s(\d))= \hat A\s(\d).$$
and in turn that  $\s(\d)=k\d$. Comparing coefficients, we have that $k=1$ and in turn that $m_i=m_{\s(i)}$. For the last statement recall from \cite{Kac} that 
$(m_0^\vee,\ldots,m_n^\vee)$ generates linearly the kernel of $\widehat A^t$, so that we can argue as above.
\end{proof}

Let $Isom(V^*)$  denote the set of isometries of $V^*$ and
$$
I(C_1)=\{\phi\in Isom(V^*)\mid \phi(C_1)=C_1\}.
$$

\begin{prop}\label{prima}$I(C_1)\cong Aut(\Pia)$.
\end{prop}
\begin{proof} Let $\nu:V^*\to V$ be the identification via the invariant form $(\cdot,\cdot)$.
Recall that $\nu(C_1)$ is the simplex with vertices $o_i,i=0,\ldots,n$, where $o_0=0,\,o_i=\varpi_i/m_i$. Given 
$\phi\in I(C_1)$, then $z=\nu\circ \phi\circ\nu^{-1}$ permutes the $o_i$'s, hence induces a permutation of $\Pia$, denoted by $f_\phi$. We claim 
that $f_\phi\in Aut(\Pia)$. 
First we prove that
\begin{equation}\label{prel}\phi(\a_i)=\frac{m_{f_\phi(i)}}{m_i}\a_{f_\phi(i)}.\end{equation}
Indeed $\a_j(o_r)=\d_{jr}\frac{1}{m_j}$; on the other hand $\phi(\a_i)(o_r)=\a_i(z^{-1}o_r)=\d_{rf_\phi(i)}\frac{1}{m_i}=\frac{m_{
f_\phi(i)}}{m_i}\d_{rf_\phi(i)}\frac{1}{m_{f_\phi(i)}}$.  

Since $\phi$ is an isometry, we have, by \eqref{prel},
$$||\a_i||^2=||\phi(\a_i)||^2=\left(\frac{m_{f_\phi(i)}}{m_i}\right)^2||\a_{f_\phi(i)}||^2.$$
But the ratio $||\a_i||^2/||\a_{f_\phi(i)}||^2$ can be just $1,2$ or $3$ (or $1/2,1/3$). Since $\frac{m_{f_\phi(i)}}{m_i}\in\mathbb Q$, the only possibility
is $1$, so that
 $m_{f_\phi(i)}=m_i$. Hence \eqref{prel} simplifies to
\begin{equation}\label{rel}\phi(\a_i)=\a_{f_\phi(i)}, \end{equation}
and in turn, since $\phi$ is an isometry, we have 
$$a_{f_\phi(i)f_\phi(j)}=\frac{2(\a_{f_\phi(i)},\a_{f_\phi(j)})}{(\a_{f_\phi(j)},\a_{f_\phi(j)})}=\frac{2(\phi\a_i,\phi\a_j)}{(\phi\a_j,\phi\a_j)}=a_{ij}.$$
We have established a  map $I(C_1)\to Aut(\Pia),\,\phi\mapsto f_\phi,$ which is clearly a group monomorphism. To prove its surjectivity, consider $f\in\ Aut(\Pia)$ and let $\phi$ denote the unique 
affine map on $V^*$ such that $\phi(\a_i)=\a_{f(i)}$. We first check that $\phi$ is an isometry. By \cite[(6.2.2)]{Kac} there exists an invariant form $\langle\cdot,\cdot\rangle$ for  which $\langle\a_i,\a_j\rangle=a_{ij}m_j(m^\vee_j)^{-1}$ so, 
 since $f\in Aut(\Pia)$, Lemma \ref{ai} and the fact that all nondegenerate invariant  bilinear symmetric forms on a simple Lie algebra are proportional show that $(\phi(\a_i),\phi(\a_j))=(\a_i,\a_j)$. 
 
 Set $z=\nu\circ \phi\circ\nu^{-1}$. Then $\a_j(z(o_i))= \phi^{-1}(\a_j)(o_i)=\d_{j, f(i)}\frac{1}{m_i}=\d_{j, f(i)}\frac{1}{m_{f(i)}}=\a_j(o_{f(i)})$. It follows that $z(o_i)=o_{f(i)}$, hence $f_\phi=f$.
  \end{proof}
 
Let $w_0$ be the longest element of $W$ and $w_0^i$ the longest element of the parabolic subgroup generated by 
$s_{\a_j},j\ne i$. Set $J=\{i\mid m_i=1\}$ and let $\Wa^e=P^\vee\rtimes W$ be the extended affine Weyl group. We let $\Wa^e$ act on $V^*$ via the identification $\nu:V^* \to V$. Set 
$$Z=\{Id_V,t_{\varpi_i}w_0^iw_0\mid i\in J\}.$$
It can be shown (cf. \cite{IM}) that $Z$ is isomorphic to the center of the connected simply connected Lie group with Lie algebra $\g$.
\begin{prop}\label{IMM}\cite[Prop. 1.21]{IM}\label{im} $Z=\{\phi\in \Wa^e\mid \phi(C_1)=C_1\}$. 
\end{prop} 

Set
$$
LI(C_1)=\{\phi\in Isom(V^*)\mid \phi(C_1)=C_1, \ \phi\text{ linear}\}.
$$
 
 \begin{prop}\label{decaut} $LI(C_1)\cong Aut(\Pi)$ and 
$Aut(\Pia)\cong I(C_1)=LI(C_1) \ltimes  Z$.
\end{prop}
\begin{proof} First remark that 
\begin{equation}\label{piislinear}
Aut(\Pi)=\{f\in Aut(\Pia)\mid f(\a_0)=\a_0\},
\end{equation}
Indeed, it is clear that an automorphism of $\Pia$ fixing $\a_0$ restricts to an automorphism of $\Pi$. Conversely  an automorphism $f$ of $\Pi$ 
 fixes $\theta$, and in turn it fixes $\a_0\in \Pia$. 
Moreover, on one hand 
\begin{align*}a_{0i}&=\frac{2(-\theta,\a_i)}{(\a_i,\a_i)}=-2\sum_{j=1}^nm_j\frac{2(\a_j,\a_i)}{(\a_i,\a_i)}=
-\sum_{j=1}^nm_j\frac{2(f(\a_j),f(\a_i))}{(f(\a_i),f(\a_i))}= \frac{2(-\theta,f(\a_i))}{(f(\a_i),f(\a_i))}\\&=a_{0f(i)},\end{align*}
on the other hand
\begin{align*}a_{i0}&=\frac{2(\a_i,-\theta)}{(\theta,\theta)}=-\sum_{j=1}^nm_j^\vee\frac{2(\a_i,\a_j)}{(\a_j,\a_j)}=
-\sum_{j=1}^nm^\vee_j\frac{2(f(\a_i),f(\a_j))}{(f(\a_j),f(\a_j))}= \frac{2(f(\a_i),-\theta)}{(\theta,\theta)}\\&=a_{f(i)0}.\end{align*}

From \eqref{piislinear} it follows that the isometry $z_f$ on $V$ induced by $f$ fixes $o_0=0$, hence it is linear.
Thus the map $f\mapsto \nu^{-1}\circ z_f \circ \nu$ establishes a homomorphism between $Aut(\Pi)$ and $LI(C_1)$. Clearly the map $\phi\mapsto f_\phi$ is its inverse when restricted to $LI(C_1)$.

Next we prove that $LI(C_1)$ and $Z$ generate
$I(C_1)$.
If $f\in Aut(\Pia)$, let $\a_i=f(\a_0)$. Then $m_i=1$, i.e., $i\in J$, and there exists
$\phi\in Z$ such that  $f_\phi(\a_i)=\a_0$, so that  $f_\phi f$ fixes $\a_0$, hence belongs to $Aut(\Pi)$. It is clear that $Z\cap LI(C_1)=\{e\}$ and that $Z$ is normal in $I(C_1)$.
\end{proof}

Set 
$$
Z_2=\{Id_V,t_{2\varpi_i}w_0^iw_0\mid i\in J\}
$$
and
$$
I(2C_1)=\{\phi\in Isom(V^*)\mid \phi(2C_1)=2C_1\}.
$$
From Proposition \ref{decaut} it is clear that 
$$
I(2C_1)=LI(C_1)\ltimes  Z_2 \cong  I(C_1)\cong Aut(\Pia).
$$
The first isomorphism is given by the identity on $LI(C_1)$ and by the map $t_{2\varpi_i}w_0^iw_0\mapsto t_{\varpi_i}w_0^iw_0$ on $Z_2$. The second isomorphism is the one we set up in Proposition \ref{prima}. 

We have therefore a natural action of $Aut(\Pia)$ on the set of alcoves in $2C_1$. By Proposition \ref{CcPp}, this action gives an action of $Aut(\Pia)$ on $\Ab$. Note that two abelian ideals $\i,\i'$ are connected by an edge in the Hasse diagram of $\Ab$ if and only if $w_\i(C_1)$ and $w_{\i'}(C_1)$ have a face in common. Hence the action of $Aut(\Pia)$ on $\Ab$ is an automorphism of the Hasse diagram (as an abstract graph).

 If $x\in Aut(\Pia)$, let us denote by $x\cdot  \i$ the action of $x$ on $ \i\in \Ab$. On the other hand, if we identify $Aut(\Pia)$ with $I(C_1)$ as in  Proposition \ref{prima}, then $Aut(\Pia)$ acts naturally on $V^*$.
\begin{prop}\label{dotaction}
If $\i\in\Ab$ and $x\in Aut(\Pia)$, 
 then 
\begin{equation}\label{azione}\langle x\cdot  \i\rangle =x(\langle\i\rangle).
\end{equation}
In particular, if $x=f_\phi$ with $\phi=t_{\varpi_i}w_0^iw_0$, then
\begin{equation}\label{e}
\langle x\cdot\i\rangle=w_0^iw_0(\langle\i\rangle)+h^\vee\omega_i.
\end{equation}
\end{prop}
\begin{proof} If $x\in Aut(\Pi)$,  then $x=f_\phi$ with $\phi\in LI(C_1)$. Since $C_{x\cdot\i}=\phi(C_\i)$ and $\phi(P)=P$,
$$\phi(\rho+\langle\i\rangle)=\phi(w_\i(\rho))=w_{x\cdot\i}(\rho)=\rho+\langle x\cdot\i \rangle.
$$
Since $\phi$ is linear, we have $\phi(\rho+\langle\i\rangle)=\phi(\rho)+\phi(\langle \i \rangle)$. Since $\phi(\rho)=\rho$, we have \eqref{azione}.

If $\phi=t_{\varpi_i}w_0^iw_0$ and $x=f_\phi$, then $C_{x\cdot\i}=t_{2\varpi_i}w_0^iw_0(C_\i)$.  As above we obtain
$$t_{2\varpi_i}w_0^iw_0(\rho+\langle\i\rangle)=\rho+\langle x\cdot\i \rangle.
$$
Remark that, under the identification of $V$ and $V^*$, 
$\varpi_i=\frac{2}{(\a_i,\a_i)}\omega_i=\frac{2}{(\theta,\theta)}\omega_i=h^\vee\omega_i$,
hence
 \begin{align*}
t_{2\varpi_i}w_0^iw_0(\rho+\langle \i\rangle)&=w_0^iw_0(\rho)+w_0^iw_0(\langle \i\rangle)+2h^\vee\omega_i\\
&=\rho+w_0^iw_0(\rho)-\rho+w_0^iw_0(\langle \i\rangle)+2h^\vee\omega_i\\
&=\rho-\langle N(w_0^iw_0)\rangle+w_0^iw_0(\langle \i\rangle)+2h^\vee\omega_i.
\end{align*}
  We now  observe that $\langle N(w_0^i w_0)\rangle= h^\vee \omega_i$. In fact, 
$N(w_0^i w_0)$ is the set of roots of the nilradical $\mathfrak n_i$ of the parabolic subalgebra defined by $\varpi_i$. It follows that $\langle N(w_0^i w_0)\rangle= x \omega_i$ for some $x\in\R$. Moreover $\dim \n_i =\langle N(w_0^i w_0)\rangle(\varpi_i)=x\omega_i (\varpi_i)=x\frac{ (\a_i,\a_i)}{4} tr(\varpi_i^2)=  x\frac{ (\theta,\theta)}{2}\dim\n_i$.
It follows that $x=h^\vee$, hence
$$
t_{2\varpi_i}w_0^iw_0(\rho+\langle \i\rangle)=\rho+w_0^iw_0(\langle \i\rangle)+h^\vee\omega_i,
$$
and, in turn,
$$
\langle x\cdot \i\rangle=w_0^iw_0(\langle \i\rangle)+h^\vee\omega_i
$$
as wished.
 \end{proof}

As an application, we recover a nice result by Suter on the Young lattice. Recall that the latter is the lattice of partitions of a natural number ordered by containment of the corresponding Young diagram. We display Young diagrams in the French way. Also recall that the hull of a  Young diagram is the minimal rectangular diagram containing it.\vskip5pt
For a positive integer $n$ let $Y_n$ be the Hasse graph for the subposet $\mathfrak Y_n$ of the 
Young  lattice corresponding to those diagrams whose hulls are contained in the staircase
diagram for the partition $(n -1, n -2, . . . , 1)$. 
\begin{theorem} \cite[Theorem 2.1]{Su} If $n \ge 3$, the dihedral group  of order
$2n$ acts faithfully on the (undirected) graph $Y_n$.
\end{theorem}
Our proof of this theorem relies on the connection between the symmetries of the  Young lattice and $Aut(\Pia)$. This connection has already been observed in \cite{S}. 

Specialize to $\g=sl(n,\C)$, and fix   as Borel subalgebra the set of lower triangular matrices. Let $e_{ij}$ denote the elementary matrices and set $\e_i(e_{hh})=\d_{ih}$. Our choice of $\b$ gives  $\Dp=\{\e_i-\e_j\in\R^n\mid i>j\}$; the corresponding simple roots are $\a_i=\e_{i+1}-\e_i$ ($i=1,\dots,n-1$). Moreover, the  positive root spaces are 
 $\g_{\e_i-\e_j}=\C e_{ij} (i>j)$. Then abelian ideals of $\b$ correspond   bijectively via
 $$\l_1\geq\ldots\geq\l_k \longleftrightarrow \sum_{h=1}^k\sum_{j=1}^{\l_h}\C \g_{\e_{n-h+1}-\e_j}$$
 to 
subspaces of strictly lower triangular matrices such that their non-zero entries form a Young diagram whose hull is contained 
in the staircase
diagram for the partition $(n -1, n -2, . . . , 1)$:\vskip5pt
{\tiny$$\left(\begin{array}{ccccccc}
0&0&0&0&0&0&0\\
0&0&0&0&0&0&0\\
0&0&0&0&0&0&0\\
\g_{\e_4-\e_1}&0&0&0&0&0&0\\
\g_{\e_5-\e_1}&0&0&0&0&0&0\\
\g_{\e_6-\e_1}&\g_{\e_6-\e_2}&0&0&0&0&0\\
\g_{\e_7-\e_1}&\g_{\e_7-\e_2}&\g_{\e_7-\e_3}&\g_{\e_7-\e_4}&0&0&0\end{array}\right)
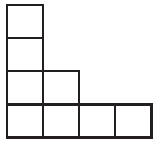
$$
\vskip5pt
\centerline{\begin{Young}
   \cr
  \cr
  &\cr
  & & &\cr
\end{Young}
}}
\vskip10pt
Let $\l(\i)$ be the diagram (or partition) corresponding to $\i\in\Ab$. Suter defines an action on $\mathfrak Y_n$ of two operators $\tau, \sigma_n$ which generate  the dihedral group  of order
$2n$. The operator $\tau$  is the involution given by flipping the diagrams along the diagonal ``South-West to North-East''; the other move is what he calls the {\it sliding move}. In formulas, if 
$\l=(\l_1,\ldots \l_m)$ is a partition whose diagram belongs to $\mathfrak Y_n$ (so that $\l_1\geq \ldots\geq\l_m,\, \l_1+m\leq n$), then 
$$\s_n(\l)=(\l_2+1,\ldots,\l_m+1,\underbrace{1,\ldots,1}_{n-m-\l_1}).$$
The term sliding comes form the following equivalent description: if $\mu=\l^t$ and $\nu= \s_n(\l)^t, $ then $\nu_1=n-\l_1-1, \nu_i=\mu_{i-1}-1,\, i\geq 2$.
\begin{prop}ÊSet $\xi=t_{\varpi_1}w_0^1w_0\in I(C_1)$. We have
$$\tau(\l(\i)) = \l(-w_0\cdot\i) , \qquad \s_n(\l(\i))=\l(f_\xi\cdot\i).
$$
In particular $\s_n$ has order $n$. Hence the action of the dihedral group generated by $\tau,\s_n$ on $\mathfrak Y_n$ is precisely the action 
of $Aut(\Pia)$ on $\Ab$. In particular it is faithful.
\end{prop}
\begin{proof} First remark that if $\l(\i)=(\l_1,\ldots,\l_m),\, \l_1+m\leq n, \i\in\Ab$, and 
$\l_i^t=(\l'_1,\ldots,\l'_r)$, then
\begin{equation}\label{peso}
\langle\i\rangle=\sum_{i=1}^m\l_i\e_{n-i+1}-\sum_{i=1}^r\l'_i\e_i.
\end{equation}
Since $-w_0(\e_i)=\e_{n-i+1}$, we have
$-w_0(\langle\i\rangle)=\sum_{i=1}^m\l_i\e_{i}-\sum_{i=1}^r\l'_i\e_{n-i+1}$, which is 
precisely $\langle\i'\rangle$, where $\i'$ is such that $\l(\i')=\tau(\l(\i))$. It follows from Proposition \ref{dotaction} that $\i'=-w_0\cdot\i$. 
Next, we compute $t_{\varpi_1}w_0^1w_0(\langle\i\rangle)$. Recall that $h^\vee\omega_1=-(n-1)\e_1+\sum_{i=2}^n\e_i$  and that $w_0^1w_0$ is the cycle $(1,2,\ldots,n)$. Hence, 
\begin{align*}t_{\varpi_1}w_0^1w_0(\langle\i\rangle)&=
\l_1e_1+\sum_{i=2}^m\l_i\e_{n-i+2}-\sum_{i=2}^{r+1}\l'_{i-1}\e_i-(n-1)\e_1+\sum_{i=2}^n\e_i\\
&=(\l_2+1)\e_n+\ldots+(\l_m+1)\e_{n-m+2}+\sum^{n-m+1}_{i=r+2}\e_i\\
&-(n-1-\l_1)\e_1-(\l'_2-1)\e_2-\dots-(\l'_r-1)\e_{r+1},
\end{align*}
which is 
precisely $\langle\i'\rangle$, where $\i'$ is such that $\l(\i')=\s_n(\l(\i))$. It follows from Proposition \ref{dotaction} that $\i'=f_\xi\cdot\i$. 
\end{proof}
\section{Symmetries of $\Ab$}
This section is devoted to the proof of Theorem \ref{T} below.
We will exploit the poset isomorphism between $\Wab$ and $\Ab$ described in Remark \ref{posetiso},
so we need to translate the action of $Aut(\Pi)$ on $\Ab$ into an action on $\Wab$.

\begin{lemma}\label{posetaction}
If $\phi\in LI(C_1)$ then $w_{f_\phi\cdot\i}=\phi w_\i\phi^{-1}$.
In particular, if $w_\i=s_{i_1}\dots s_{i_r}$ is a reduced expression for $w_\i$ and $f\in Aut(\Pi)$, then $s_{f(i_1)}\dots s_{f(i_r)}$ is a reduced expression for $w_{f\cdot\i}$.
\end{lemma}
\begin{proof}
It is enough to observe that 
$$
w_{f_\phi\cdot\i}(C_1)=\phi(w_\i(C_1))=\phi w_\i\phi^{-1}(\phi(C_1))=\phi w_\i\phi^{-1}(C_1).
$$
\end{proof}

We can define a labeling on the edges of Hasse diagram $H_{\Ab}$ of $\Ab$ by the following procedure: if $u,v\in \mathcal W, u<w$ are adjacent in $H_\Ab$, then 
 $v=us_i$. We assign the label $i$ to the edge $u \to us_i$. We number diagrams as in \cite{S}.
\begin{lemma}\label{fund} If $w\in\Waab$, then any reduced expression of $w$ avoids substrings of the form $s_\a s_\beta s_\a$ except when $\a$ is a long simple root and  $\beta$ is a short simple root. 
\end{lemma}
\begin{proof} In the contrary case, there exists a reduced expression of the form $w=w's_\a s_\beta s_\a w''$, with either $\a,\beta$ of the same length or $\a$ short and $\beta$ long.
 Remark that in both cases $s_\beta(\a)=\a+\beta$, so that 
$s_\a s_\beta(\a)= s_\a(\beta)-\a$. Then $N(w)$ contains 
$w'(\a),w'(s_\a(\beta)),w'(s_\a(\beta)-\a)$, against the fact that $w$ encodes an abelian ideal.
\end{proof}
\begin{cor}\label{bendef} The only way to change a reduced expression for $w\in \Waab$ is to  switch  two consecutive commuting simple reflections.
In particular, given $w\in\Waab$,  any  reduced expression of $w$ contains the same number of occurrences of a simple reflection.
\end{cor}
\begin{proof} It is well-known (see \cite{Mat}) that in a Coxeter group it is possible to pass from a reduced expression of an element to another by switching commuting generators or by applying braid relations. If $w\in\Waab$, the latter moves are forbidden by the above Lemma. Indeed, let $m_{\a,\beta}$ be the order of $s_\a s_\beta$. If $m_{\a,\beta}=3$, then $\a,\beta$ have the same length and the 
braid relation $s_\a s_{\beta}s_\a=s_{\beta}s_\a s_{\beta}$ is forbidden by the Lemma. If $m_{\a,\beta}>3$, then in a braid relation the forbidden pattern appears.
\end{proof}
\begin{rem}\label{previous}
Observe that if $v$ has at least two reduced expressions, then in the order ideal generated by $v$  a diamond
$$
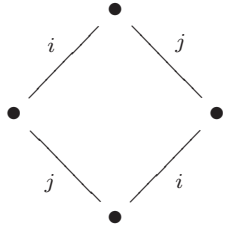
\xymatrix@R+3pt@C+1pt{ 
&&&{\bullet}\ar@{-}[dr]^j\ar@{-}[dl]_i&&&\\ 
&&{\bullet}\ar@{-}[dr]_j&&{\bullet}\ar@{-}[dl]^i&\\  
&&&{\bullet}&&\\ 
}
$$
appears. We next show that diamonds  in $H_{\Ab}$ occur precisely in this situation.
Indeed, if a diamond
$$
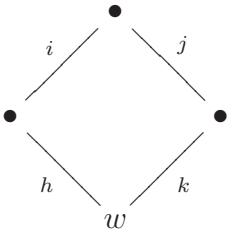
\xymatrix@R+3pt@C+1pt{ 
&&&{\bullet}\ar@{-}[dr]^j\ar@{-}[dl]_i&&&\\ 
&&{\bullet}\ar@{-}[dr]_h&&{\bullet}\ar@{-}[dl]^k&\\  
&&&{w}&&\\ 
}
$$
occurs in $H_{\Ab}$, then 
$i=k$, $j=h$, and $s_is_j=s_js_i$. This follows observing that $ws_hs_i=ws_ks_j$ obviously implies $s_hs_i=s_ks_j$ and the latter relation 
holds if and only if $h=j, k=i.$
 \end{rem}
\
\begin{rem}\label{51}
The minimal abelian ideal is $\{0\}$, and there is just one 1-dimensional ideal, spanned by a highest root vector. Both are contained in any other abelian ideal in $\b$. In terms of alcoves, the first 
corresponds to $C_1$ and the second to $s_{0}(C_1)$.
Thus the Hasse diagram $H_{\Ab}$ starts with a chain
\begin{equation}\label{f1}
\xymatrix@R+2pt@C+1pt{ 
{s_0}\ar@{-}[d]_0\\  
{e}\\ 
}
\end{equation}
$e$ being the neutral element of $\Wa$.
\end{rem}

Set $\Waab_k=\{w\in\Waab\mid \ell(w)\le k\}$. Denote by $\Pi'$  the set of labels $i$ appearing in $H_{\Ab}$. Thus $\Pi'$ is the set of $i$ such that $s_i$ occurs in  a reduced expression of an element of $\Waab$. Indeed, $\Pi=\Pi'$ in any case except type $C$, in which the simple reflection corresponding to the long simple root in the (finite) Dynkin diagram does not appear.
\vskip5pt
Let $Aut(\Ab)$  be the set of poset automorphisms of $\Ab$.
\begin{theorem}\label{T} If $\g$ is not of type $C_3$, then $Aut(\Ab)\cong Aut(\Pi)$.
\end{theorem}
Type $C_3$ is dealt with in \cite{S}: the picture at p. 213 shows that  $Aut(\Ab)\cong\ganz_2$; on the other hand $Aut(\Pi)$ is trivial. We exclude this case from now on.\par Before tackling  the proof of the Theorem, we single out the  low rank cases $C_2, A_3$, which will be referred to in the following. \vskip5pt
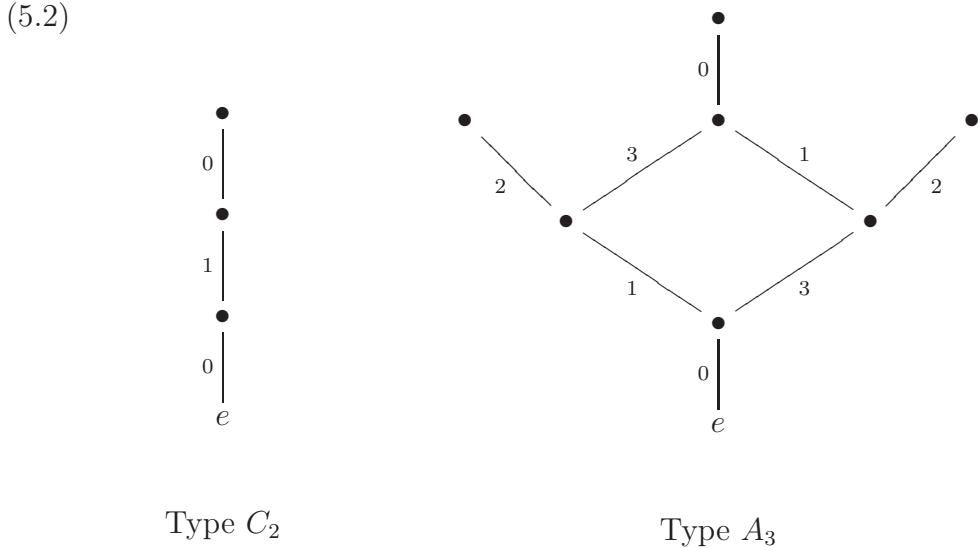
\begin{equation}\label{f2}
{\xymatrix@R+2pt@C+1pt{ 
\\
{\bullet}\ar@{-}[d]_0\\  
{\bullet}\ar@{-}[d]_1\\ 
{\bullet}\ar@{-}[d]_0\\  
{e}\\ 
{\text{Type}\ C_2}\\
}
}\qquad\qquad\quad{
\xymatrix@R+2pt@C+1pt{ &&{\bullet}\ar@{-}[d]_{0}&\\  
{\bullet}\ar@{-}[dr]_{2}&&{\bullet}\ar@{-}[dr]^{1}\ar@{-}[dl]_{3}&&{\bullet}\ar@{-}[dl]^{2}\\ 
&{\bullet}\ar@{-}[dr]_{1}&&{\bullet}\ar@{-}[dl]^{3}&\\  
&&{\bullet}\ar@{-}[d]_{0}&\\ 
&&{e}&\\
& &{\text{Type}\ A_3}\\
}}
\end{equation}
\vskip10pt

\par
Let $\Waab'_h$ be the subset of $\Waab_{h}$ consisting of elements which have at least two reduced expressions.
\begin{lemma}\label{prell} Let $\s\in Aut(\Ab)$ be such that $\s_{|\Waab_{h-1}}=Id$.  Then $\s(w)=w$ for any $w\in\Waab'_h$.\end{lemma}
\begin{proof} 
Since $\Waab'_1=\emptyset$, we can clearly assume $h>1$. 
 Let $w$ be a node in $\Waab'_{h}$. By the very definition of $\Waab'_{h}$, the order ideal generated by $w$ contains a subdiagram of the form 
$$
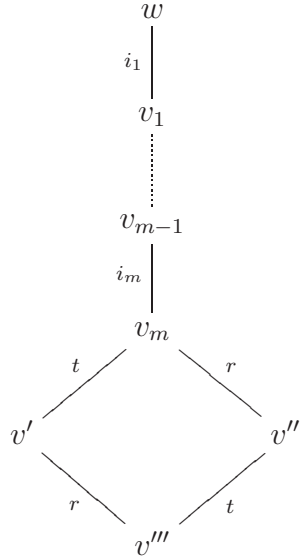
\xymatrix@R+2pt@C+1pt{ &{w}\ar@{-}[d]_{i_1}&\\
&{v_1}\ar@{.}[d]&\\ 
&{v_{m-1}}\ar@{-}[d]_{i_m}&\\  
&{v_m}\ar@{-}[dr]^{r}\ar@{-}[dl]_{t}&\\ 
{v'}\ar@{-}[dr]_{r}&&{v''}\ar@{-}[dl]^{t}\\  
&{v'''}&\\ 
}
$$
Choose a subdiagram with minimum $m$. If $m=0$,  we have a diamond
$$
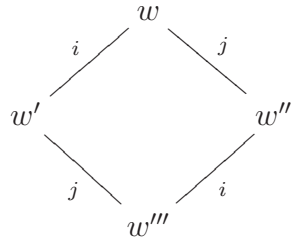
\xymatrix@R+2pt@C+1pt{ 
&&&{w}\ar@{-}[dr]^j\ar@{-}[dl]_i&&&\\ 
&&{w'}\ar@{-}[dr]_j&&{w''}\ar@{-}[dl]^i&\\  
&&&{w'''}&&\\ 
}
$$
This is mapped by $\s$ into a diamond 
$$
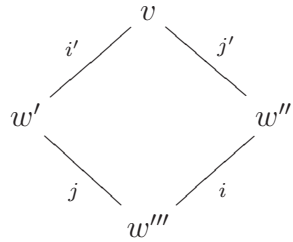
\xymatrix@R+2pt@C+1pt{ 
&{v}\ar@{-}[dr]^{j'}\ar@{-}[dl]_{i'}&\\ 
{w'}\ar@{-}[dr]_{j}&&{w''}\ar@{-}[dl]^{i}\\  
&{w'''}&\\ 
}
$$
By Remark  \ref{previous}, we have $i'=i$ and $j'=j$. It follows that $v=\s(w)=w's_{i}=w$.

Assume now that $m=1$, so that  $$
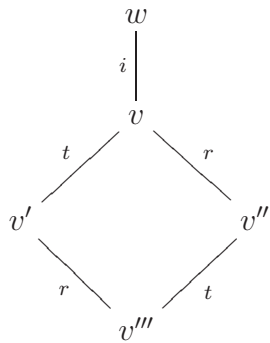
\xymatrix@R+2pt@C+1pt{ &{w}\ar@{-}[d]_i&\\  
&{v}\ar@{-}[dr]^{r}\ar@{-}[dl]_{t}&\\ 
{v'}\ar@{-}[dr]_{r}&&{v''}\ar@{-}[dl]^{t}\\  
&{v'''}&\\ 
}
$$
with $s_i s_{r} \ne s_{r} s_{i}$, and $s_i s_{t} \ne s_{t} s_{i}$. Thus $\a_i,\a_r,\a_t$ form an irreducible subsystem of $\Da$ of rank $3$ with $\a_r,\a_t$ orthogonal.  Applying $\s$ to the above diagram we have
$$
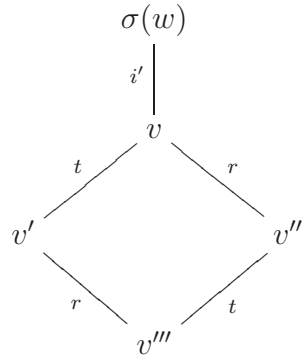
\xymatrix@R+2pt@C+1pt{ &{\s(w)}\ar@{-}[d]_{i'}&\\  
&{v}\ar@{-}[dr]^{r}\ar@{-}[dl]_{t}&\\ 
{v'}\ar@{-}[dr]_{r}&&{v''}\ar@{-}[dl]^{t}\\  
&{v'''}&\\ 
}
$$
Thus $\a_{i'}$ is not orthogonal to both $\a_t$ and $\a_r$. Except in type $A_3$, this implies that $i'=i$. Indeed both $\a_{i'}$ and $\a_i$ are connected to $\a_t$, $\a_r$ in $\Pia$. Thus, if $i\ne i'$, $\a_t$, $\a_r$, $\a_i$, $\a_{i'}$ form a cycle, hence $\Pia$ is of type $\widehat A_3$, and we are done by looking at \eqref{f2}.
Next we assume $m=2$:
$$
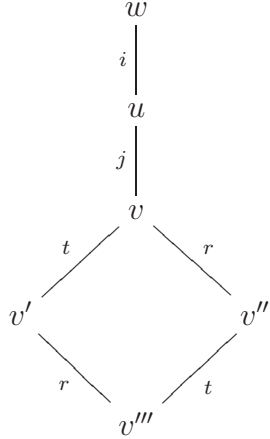
\xymatrix@R+2pt@C+1pt{ &{w}\ar@{-}[d]_i&\\ 
&{u}\ar@{-}[d]_{j}&\\  
&{v}\ar@{-}[dr]^{r}\ar@{-}[dl]_{t}&\\ 
{v'}\ar@{-}[dr]_{r}&&{v''}\ar@{-}[dl]^{t}\\  
&{v'''}&\\ 
}
$$
The automorphism $\s$ maps this configuration to
$$
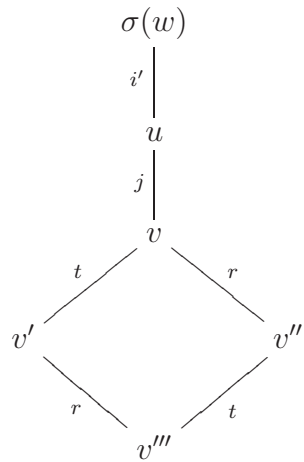
\xymatrix@R+2pt@C+1pt{ &{\s(w)}\ar@{-}[d]_{i'}&\\ 
&{u}\ar@{-}[d]_{j}&\\  
&{v}\ar@{-}[dr]^{r}\ar@{-}[dl]_{t}&\\ 
{v'}\ar@{-}[dr]_{r}&&{v''}\ar@{-}[dl]^{t}\\  
&{v'''}&\\ 
}
$$
Assume first that $\a_{i}=\a_t$ or $\a_{i}=\a_r$. For simplicity assume $\a_{i}=\a_t$, then $\a_j$ must be short and $\a_{i}$ is long. Thus there are only two roots connected to $\a_j$. Since $\a_{i'}$ is connected to $\a_j$, we must have that either $\a_{i'}=\a_t=\a_i$ or $\a_{i'}=\a_r$. In the latter case $\a_r$ must be long, for, otherwise, $s_rs_js_{i'}=s_rs_js_{r}$  is forbidden  braid (see Lemma \ref{fund}). This implies that we are in type $\widehat C_2$, and we are done again by looking at \eqref{f2}.\par
We can therefore assume that $\a_i\ne \a_t$ and $\a_i\ne \a_r$. Since $\a_j$ is not orthogonal to $\a_i$, $\a_r$, $\a_t$, there  are at least  three vertices stemming from $\a_j$ in $\Pia$. If $i'=t$ or $i'=r$, then the braid $s_rs_js_r$ or $s_ts_js_t$ would occur in a reduced expression for $\s(w)$, which is impossible by Lemma \ref{fund}. This implies that either $i'=i$ or there are four edges stemming from $\a_j$, i. e. we are in type $D_4$. This latter case is handled by a direct inspection: the Hasse diagram for type $D_4$ is given in \cite[p. 217]{S}, and in this case one can check that $i=i'=0$.

We now assume that $m\ge 3$.
First assume $\a_{i_m}$ long and $\a_{i_m}\ne\a_{i_2}$. This implies that the roots $\a_t$, $\a_r$, and $\a_{i_{m-1}}$ are distinct and all connected to $\a_{i_m}$. Thus $\a_{i_m}$ is a vertex of degree at least three in the Dynkin diagram. The automorphism $\s$ maps this configuration to
$$
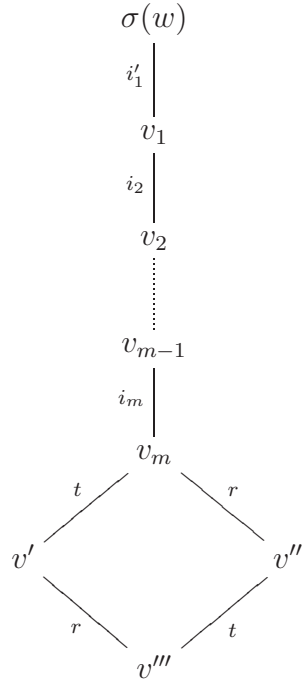
\xymatrix@R+2pt@C+1pt{ &{\s(w)}\ar@{-}[d]_{i'_1}&\\ 
&{v_1}\ar@{-}[d]_{i_2}&\\
&{v_2}\ar@{.}[d]&\\ 
&{v_{m-1}}\ar@{-}[d]_{i_m}&\\  
&{v_m}\ar@{-}[dr]^{r}\ar@{-}[dl]_{t}&\\ 
{v'}\ar@{-}[dr]_{r}&&{v''}\ar@{-}[dl]^{t}\\  
&{v'''}& 
}
$$
If $i_1\ne i_1'$, then $\a_{i_1}$, $\a_{i_3}$, and $\a_{i'_1}$ are all connected to $\a_{i_2}$. If they are not all distinct then $\a_{i_2}$ is short and $\a_{i_3}$ is long. Since there is a degree three vertex in the diagram, we are in type $\widehat B_n$ and both $\a_{i_1}$ and $\a_{i'_1}$ are connected to the unique short simple root. Hence $i_1=i'_1$ as desired. 
\par
We can therefore assume that  $\a_{i_1}$, $\a_{i_3}$, and $\a_{i_1'}$ are pairwise distinct. It follows that there are at least three vertices stemming from $\a_{i_2}$ in $\Pia$. Recall that we assumed that $\a_{i_m}\ne \a_{i_2}$, so there are two vertices of degree three in the Dynkin diagram. Thus we are in type $\widehat D_{n}$ with $n\ge 5$: indeed we claim that we are in the following situation.
$$
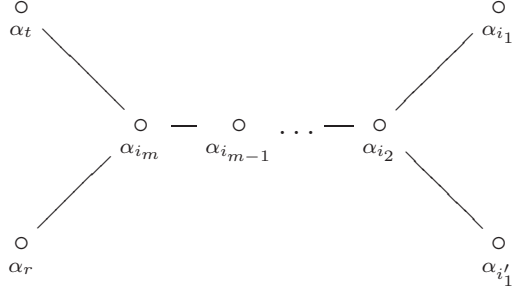
\xymatrix@!=4ex{
 \mathop\circ\limits_{\a_t}\ar@{-}[dr] & & &  & \mathop\circ\limits_{\alpha_{i_1}}\\
&\mathop\circ\limits_{\a_{i_m}} \ar@{-}[r]&\mathop\circ\limits_{\a_{i_{m-1}}} \ar@{-}[r] \dots \ar@{-}[r] &{\mathop\circ\limits_{\a_{i_{2}}}}\ar@{-}[ur]\ar@{-}[dr]\\
 \mathop\circ\limits_{\a_r}\ar@{-}[ur]&& &  & \mathop\circ\limits_{\alpha_{i_1'}}
}
$$
In fact, $\a_{i_{m-1}}$ is connected to $\a_{i_m}$ and cannot be $\a_t$ or $\a_r$ for, in such a case, the braid $s_ts_{i_m}s_t$ or $s_rs_{i_m}s_r$ would occur in a reduced expression for $w$. The same argument shows that $\a_{i_{m-j}}$ is connected to $\a_{i_{m-j+1}}$ and cannot be $\a_{i_{m-j+2}}$. 
We want to prove that, since $w=v'''s_ts_rs_{i_m}\dots s_{i_2}s_{i_1}$ is in $\Waab$, then $w'=v'''s_ts_rs_{i_m}\dots s_{i_2}s_{i_1'}\not\in\Waab$. First observe that $v'''\ne e$. This is so because, by Remark \ref{51}, the Hasse diagram does not start with a diamond.
Let $j$ be the label of an edge reaching $v'''$. We now prove that $j=i_m$. If $j=i_a$ with $1<a<m$,  then the braid $s_{i_a}s_{i_{a+1}}s_{i_a}$ would occur in a reduced expression of $w$. If $j=i_1,i_1'$, then the braids $s_{i_1}s_{i_{2}}s_{i_1}$, $s_{i'_1}s_{i_{2}}s_{i'_1}$ would occur in a reduced expression of $w,w'$ respectively. Since obviously $j\ne r,t$ we have $j=i_m$.
\par
Repeating this argument we find that
$$w=us_{i_2}s_{i_3}\dots s_{i_m}s_rs_ts_{i_m}\dots s_{i_2}s_{i_1},\ w'=us_{i_2}s_{i_3}\dots s_{i_m}s_rs_ts_{i_m}\dots s_{i_2}s_{i'_1}.$$
Note that $u\ne e$, since $s_{i_2}\ne s_0$. Write $u=u's_j$. The above argument shows that $j=i_1$ or $j=i'_1$. If $j={i'_1}$, then 
$s_{i'_1}s_{i_2}s_{i_3}\dots s_{i_m}s_rs_ts_{i_m}\dots s_{i_2}(\a_{i_1})=\d+\a_{i_1'}$, thus both $u'(\a_{i_1'})$ and $\d+u'(\a_{i_1'})$ are in $N(w)$, which is absurd, since $w\in\Waab$. It follows that $j=i_1$, but then 
$u'(\a_{i_1})$ and $\d+u'(\a_{i_1})$ are both in $N(w')$ so $w'\not\in \Waab$.

Assume now that $i_2=i_m$. If $\a_{i_j}$ are all long roots then $i_m\ne i_{m-2}$, otherwise we would have a forbidden braid. So the Dynkin diagram is
$$
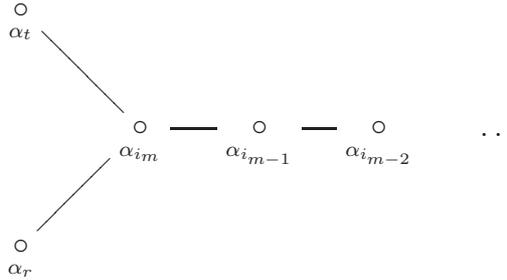
\xymatrix@!=4ex{
 \mathop\circ\limits_{\a_t}\ar@{-}[dr] & & &  & \\
&\mathop\circ\limits_{\a_{i_m}} \ar@{-}[r]&\mathop\circ\limits_{\a_{i_{m-1}}} \ar@{-}[r] &{\mathop\circ\limits_{\a_{i_{m-2}}}}&\dots\\
 \mathop\circ\limits_{\a_r}\ar@{-}[ur]&& &  & }
$$
By the same argument $i_{m-3}\ne i_{m-1},i_{m-4}\ne i_{m-2},\dots, i_{4}\ne i_{2}$ so the Dynkin diagram is
$$
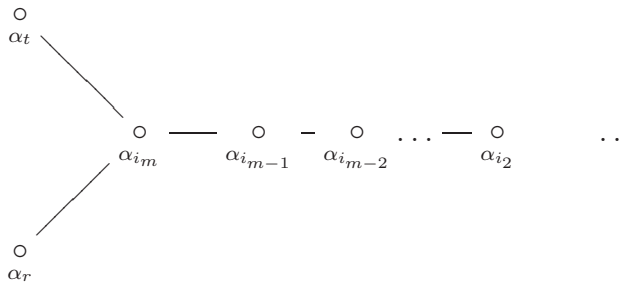
\xymatrix@!=4ex{
 \mathop\circ\limits_{\a_t}\ar@{-}[dr] & & &  & \\
&\mathop\circ\limits_{\a_{i_m}} \ar@{-}[r]&\mathop\circ\limits_{\a_{i_{m-1}}} \ar@{-}[r] &{\mathop\circ\limits_{\a_{i_{m-2}}}\dots}&{\mathop\circ\limits_{\a_{i_{2}}}}\ar@{-}[l]&\dots\\
 \mathop\circ\limits_{\a_r}\ar@{-}[ur]&& &  & }
$$
contradicting the hypothesis that $i_m=i_2$. This implies that there is $k$ such that $\a_{i_k}$ is a short root. Thus we are in type $\widehat B_n$. We claim that we are in the following situation:
$$
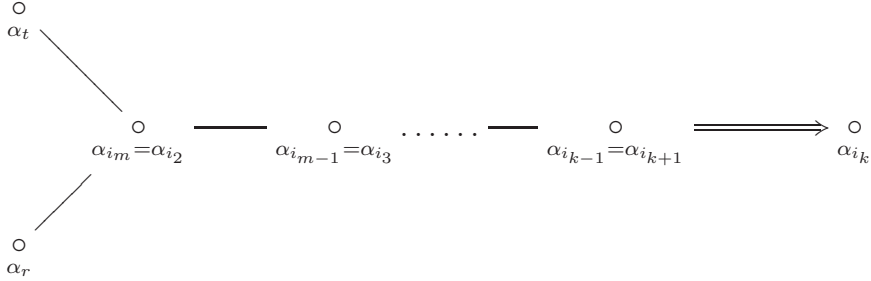
\xymatrix@!=4ex{
 \mathop\circ\limits_{\a_t}\ar@{-}[dr] & & &  &\\
&\mathop\circ\limits_{\a_{i_m}=\a_{i_2}} \ar@{-}[rr]&&\mathop\circ\limits_{\a_{i_{m-1}}=\a_{i_3}} \ar@{-}[r] \dots\dots  &&{\mathop\circ\limits_{\a_{i_{k-1}}=\a_{i_{k+1}}}}\ar@{=>}[rr]\ar@{-}[ll]&&{\mathop\circ\limits_{\a_{i_{k}}}}\\
 \mathop\circ\limits_{\a_r}\ar@{-}[ur]&& &  }
$$
Indeed, let $i_k$ be the first occurrence of the short simple root. It follows that, since $\a_{i_{k-1}}$ is connected to $\a_{i_k}$, we have that $i_{k-1}=i_{k+1}$. Since $\a_{i_{k+2}}$ is connected to $\a_{i_{k+1}}=\a_{i_{k-1}}$ and the braid $s_{i_k}s_{i_{k+1}}s_{i_k}$ is forbidden, we see that $i_{k-2}=i_{k+2}$. The same argument shows that $i_{k+j}=i_{k-j}$ for $j=0,\dots,k-1$.
Thus $w=v'''s_rs_ts_{i_2}s_{i_3}\dots s_{i_k}\dots s_{i_2}s_{i_1}$ and $i_1=r$ or $i_1=t$. Assume for simplicity $i_1=r$. Then both $v'''(\a_t)$ and $\d+v'''(\a_t)$ are in $N(w)$ and this is impossible.

Assume now $\a_{i_m}$ short. Note that $\a_{i_{m-1}}$, $\a_r$, $\a_t$ cannot be pairwise distinct for, otherwise, $\a_{i_{m}}$ would be a node of degree three in a non simple laced  Dynkin diagram, hence $\a_{i_m}$ would be long. It follows that $\a_{i_{m-1}}=\a_t$ or $\a_{i_{m-1}}=\a_r$. Assume for simplicity that $\a_{i_{m-1}}=\a_t$. Since $\a_r$ is orthogonal to $\a_{t}$ we have the following situation
$$
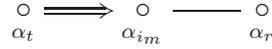
\xymatrix@!=4ex{
{\mathop\circ\limits_{\a_t}}\ar@{=>}[r]&{\mathop\circ\limits_{\a_{i_{m}}}}\ar@{-}[r]&{\mathop\circ\limits_{\a_r}
} }
$$

Now $\a_{i_{m-2}}$  is connected to $\a_{i_{m-1}}$ and it is not $\a_r$ for, otherwise, there would be a forbidden braid. It follows that we are in type $\widehat F_4$ and $t=i_{m-1}=2$, $i_{m-2}=1$, $i_m=3$, and $r=4$. This does not happen, as a direct inspection of the Hasse diagram shows\footnote{Indeed an argument avoiding the inspection could be provided, but looking at the Hasse diagram is certainly handier in this case.} (see \cite[p. 218]{S}).
\end{proof}

If $\s\in Aut(\Ab)$, let $h_\s$ be the maximal $h\in\nat$ such that $\s_{|{\Waab_{h}}}=Id$. 
If $\s=Id$ then we set $\Waab_{h_\s+1}=\Waab_{h_\s}=\Waab$.

Lemma \ref{posetaction} implies clearly that $Aut(\Pi)$ acts by poset automorphisms. Note also that $f\in Aut(\Pi)$ acts on $H_\Ab$ as an automorphism of labelled graphs, and that the induced map on labels is $f$ itself restricted to $\Pi'$. This fact will be used without comment in the proof of the following result.

\begin{lemma}\label{hsigma}
If $\s\in Aut(\Ab)$ then $\s_{|\Waab_{h_\s+1}}\in Aut(\Pi)_{|\Waab_{h_\s+1}}$.
\end{lemma}
\begin{proof}
Since the poset starts as in \eqref{f1}, we have that  $h_\s\ge1$. Clearly we can assume $\s\ne Id$.
Let  $w\in\Waab_{h_\s+1}$ be such that  $\s(w)\ne w$. Then, by Lemma \ref{prell}, we have that  $w\notin\Waab'_{h_\s+1}$. Hence the interval from $e$ to $w$ is 
 a chain:
$$
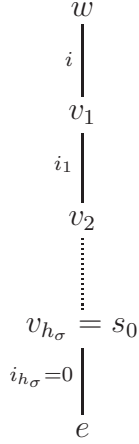
\xymatrix@R+2pt@C+1pt{ 
{w}\ar@{-}[d]_{i}\\ 
{v_1}\ar@{-}[d]_{i_1}\\
{v_2}\ar@{.}[d]\\ 
{v_{h_\s}=s_0}\ar@{-}[d]_{i_{h_\s}=0}\\
{e}  
}
$$
Since $w\in \Waab_{h_\s+1}$, the automorphism $\s$ maps the above chain to
$$
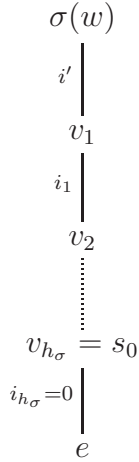
\xymatrix@R+2pt@C+1pt{ 
{\s(w)}\ar@{-}[d]_{i'}\\ 
{v_1}\ar@{-}[d]_{i_1}\\
{v_2}\ar@{.}[d]\\ 
{v_{h_\s}=s_0}\ar@{-}[d]_{i_{h_\s}=0}\\
{e}  
}
$$
with $i\ne i'$. 

Let us discuss the case $h_\s=1$. If $h_\s=1$, then $w=s_0s_i$. If $\s(w)=s_0s_{i'}$ with $i\ne i'$  then there are two simple roots connected to $\a_0$. This happens only in type $\widehat A_n$ with ${i,i'}={1,n}$. In this case $\Waab_2$ is 
$$
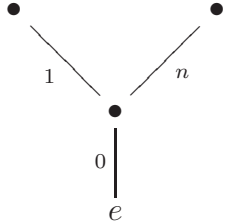
\xymatrix@R+2pt@C+1pt{ 
{\bullet}\ar@{-}[dr]_{1}&&{\bullet}\ar@{-}[dl]^{n}\\ 
&{\bullet}\ar@{-}[d]_{0}&\\
&{e}&  
}
$$
and $Aut(\Waab_2)=Aut(\Pi)_{|\Waab_2}$.
We can therefore assume that $h_\s\ge2$.

Assume first that $\a_{i'}$, $\a_{i_2}$, $\a_i$ are not pairwise distinct. For simplicity assume $\a_i=\a_{i_2}$. This implies that  $\a_{i_1}$ is short and $\a_i$ is long for, otherwise, we would have a forbidden braid. If $\a_{i'}$ is also long, then we are in type $\widehat C_2$. By looking at \eqref{f2}, we see that, in this case, $Aut(\Ab)=\{Id\}$. We therefore have that $\a_{i'}$ short. If $h_\s=2$ (so that $i=i_2=0$), then we are in type $\widehat C_n$ with $n\ge 4$. (Recall that we are excluding type $\widehat C_3$). In this case the Hasse diagram of $\Waab_4$ is 
$$
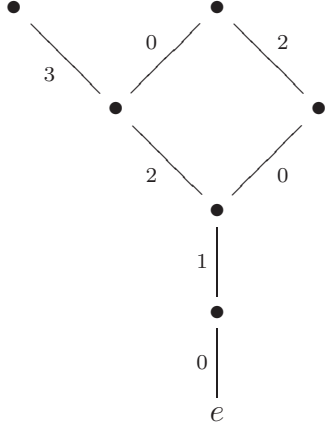
\xymatrix@R+2pt@C+1pt{ 
{\bullet}\ar@{-}[dr]_{3}&&{\bullet}\ar@{-}[dl]_{0}\ar@{-}[dr]^{2}\\ 
&{\bullet}\ar@{-}[dr]_{2}&&{\bullet}\ar@{-}[dl]^{0}\\
&&{\bullet}\ar@{-}[d]_1\\ 
&&{\bullet}\ar@{-}[d]_{0}\\
&&{e}  
}
$$
From this graph we see that $\s_{|\Waab_3}=Id$. Thus $h_\s>2$. Since $i_3\ne i_1$, we see that we are in type $\widehat F_4$, $h_\s=4$, and  $w=s_0s_1s_2s_3s_2$, but then $w\not\in\Waab$.

We can therefore assume that
$\a_{i}$, $\a_{i'}$, and $\a_{i_2}$ are pairwise distinct.  Thus $\a_{i_1}$ is a node of degree at least three in the Dynkin diagram. In particular $\a_{i_1}$ is a long root.

Assume that  there is $j>1$ such that $\a_{i_j}$ is short. Since there is a triple node and the diagram is not simply laced, we are  in type $\widehat B_n$: indeed we claim that we are in the following situation
$$
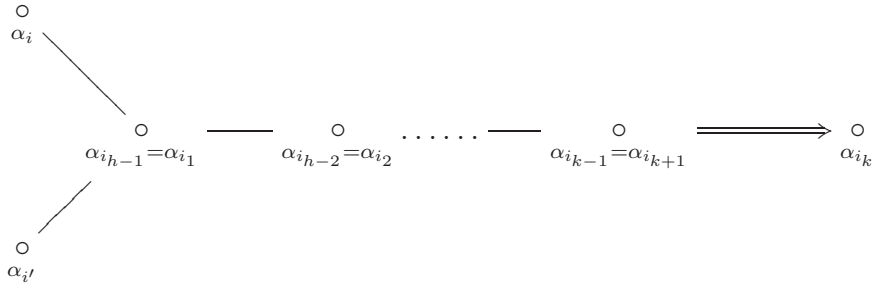
\xymatrix@!=4ex{
 \mathop\circ\limits_{\a_i}\ar@{-}[dr] & & &  &\\
&\mathop\circ\limits_{\a_{i_{h-1}}=\a_{i_1}} \ar@{-}[rr]&&\mathop\circ\limits_{\a_{i_{h-2}}=\a_{i_2}} \ar@{-}[r] \dots\dots  &&{\mathop\circ\limits_{\a_{i_{k-1}}=\a_{i_{k+1}}}}\ar@{=>}[rr]\ar@{-}[ll]&&{\mathop\circ\limits_{\a_{i_{k}}}}\\
 \mathop\circ\limits_{\a_{i'}}\ar@{-}[ur]&& &  }
$$
and $i=i_h=0$ or $i'=i_h=0$. Indeed, let $i_k$ be the first occurrence of the short simple root. It follows: since $\a_{i_{k-1}}$ is connected to $\a_{i_k}$, then $i_{k-1}=i_{k+1}$. Since $\a_{i_{k+2}}$ is connected to $\a_{i_{k+1}}=\a_{i_{k-1}}$ and the braid $s_{i_k}s_{i_{k+1}}s_{i_k}$ is forbidden, we see that $i_{k-2}=i_{k+2}$. The same argument shows that $i_{k+j}=i_{k-j}$ for $j=0,\dots,k-1$.
Thus $w=s_0s_{i_1}s_{i_2}\dots s_{i_k}\dots s_{i_1}s_0$ or $w=s_0s_{i_1}s_{i_2}\dots s_{i_k}\dots s_{i_1}s_1$. In the second case we have that   $\d+\a_0$ is in $N(w)$ and this is impossible.
Thus $i=0$ and $i'=1$. But then  we have $\s(w)=s_0s_{i_1}s_{i_2}\dots s_{i_k}\dots s_{i_1}s_1\not\in \Waab$. 

 We can therefore deduce that all
the roots $\a_{i_j}$ are long roots.  This implies that $i_j\ne i_{j-2}$ for all $j>2$, otherwise we would have forbidden braids. So $\{\a_i,\a_{i'}\}\cup\{\a_{i_j}\mid j=1,\dots,h\}$ 
 form a subdiagram of type
$$
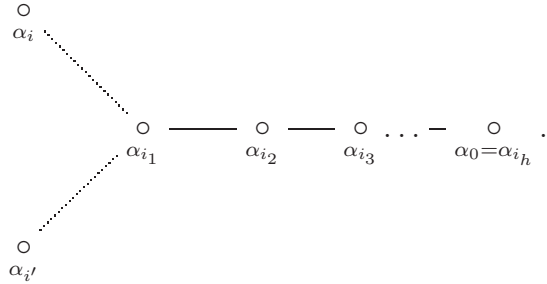
\xymatrix@!=4ex{
 \mathop\circ\limits_{\a_i}\ar@{.}[dr] & & &   \\
&\mathop\circ\limits_{\a_{i_1}} \ar@{-}[r]&\mathop\circ\limits_{\a_{i_{2}}} \ar@{-}[r] &{\mathop\circ\limits_{\a_{i_{3}}}\dots}&{\mathop\circ\limits_{\a_{0}=\a_{i_h}}.}\ar@{-}[l]\\
 \mathop\circ\limits_{\a_{i'}}\ar@{.}[ur]&& &   }
$$
The dotted edges may be multiple. This is possible in types $\widehat B_n (n\ge 3), \widehat D_n (n\ge 4),  \widehat E_n, (n=6,7,8).$\par
{\bf Case 1: type $\widehat B_n$}. Then $h_\s=2$ and $\{i,i'\}=\{1,3\}$. The Hasse graph of $\Waab_4$ is 
$$
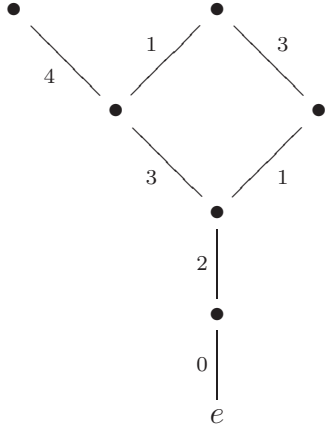
\xymatrix@R+2pt@C+1pt{ 
{\bullet}\ar@{-}[dr]_{4}&&{\bullet}\ar@{-}[dl]_{1}\ar@{-}[dr]^{3}\\ 
&{\bullet}\ar@{-}[dr]_{3}&&{\bullet}\ar@{-}[dl]^{1}\\
&&{\bullet}\ar@{-}[d]_2\\ 
&&{\bullet}\ar@{-}[d]_{0}\\
&&{e}  
}
$$
if $n\ge4$ (if $n=3$ just replace the label $4$ by $2$).
%
%
%
%
From these graphs we see that $\s_{|\Waab_3}=Id$.

{\bf Case 2: type $\widehat D_n$.} Either $h_\s=2$ or  $h_\s=n-2$.

If $h_\s=2$ then $w=s_0s_2s_i$ with $\a_i$ that ranges over the nodes connected to $\a_{2}$ and different from $\a_0$.
If $n=4$ we are done because
$\Waab_3$ is 
$$
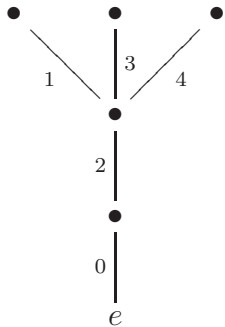
\xymatrix@R+2pt@C+1pt{ 
{\bullet}\ar@{-}[dr]_{1}&{\bullet}\ar@{-}[d]^{3}&{\bullet}\ar@{-}[dl]^{4}\\ 
&{\bullet}\ar@{-}[d]_{2}&\\
&{\bullet}\ar@{-}[d]_{0}&\\
&{e}&  
}
$$
and $Aut(\Waab_3)=Aut(\Pi)_{|\Waab_3}$. 

If $h_\s=2$ and $n>4$, then 
 the Hasse graph of $\Waab_4$ is 
$$
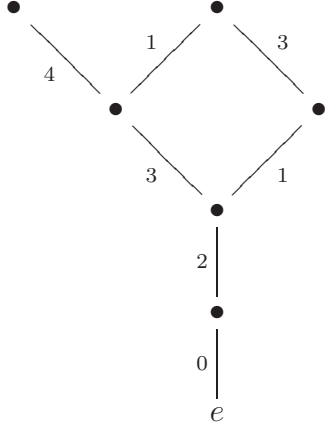
\xymatrix@R+2pt@C+1pt{ 
{\bullet}\ar@{-}[dr]_{4}&&{\bullet}\ar@{-}[dl]_{1}\ar@{-}[dr]^{3}\\ 
&{\bullet}\ar@{-}[dr]_{3}&&{\bullet}\ar@{-}[dl]^{1}\\
&&{\bullet}\ar@{-}[d]_2\\ 
&&{\bullet}\ar@{-}[d]_{0}\\
&&{e}  
}
$$
hence $Aut(\Ab)_{|\Waab_3}=\{Id\}$ so this case does not occur.

If $h_\s=n-2$, we may assume $n\ge 5$ and, by our analysis, there are only two elements that can be moved by $\s$:  these are $w=s_0s_{2}\dots s_{n-2} s_{n-1}$ and $w'=s_0s_{2}\dots s_{n-2} s_{n}$, thus $\s$ must exchange them and fix all other elements of $\Waab_{n-1}$. We need to check that $\s_{|\Waab_{n-1}}\in Aut(\Pi)_{|\Waab_{n-1}}$. To this end, it is enough to show that the only elements of  $\Waab_{n-1}$ containing $s_n$ or $s_{n-1}$ in a reduced expression   are $w$ and $w'$. This clearly concludes the proof in this case, for, then, $\s_{|\Waab_{n-1}}= \s'_{|\Waab_{n-1}}$ with $\s'\in Aut(\Pi)$ exchanging $\a_{n-1}$ and $\a_n$. 
Let $v\in\Waab_{n-1}$ contain $s_{n-1}$. Observe that any reduced expression of $v$ starts with $s_0s_2$. Also, the simple reflections $s_3,s_4,\ldots,s_{n-1}$ have to appear, and to appear exactly in this 
order. Otherwise,  let $i$ be the place where  the first violation occurs: then $v=s_0 s_2\cdots s_{i-1}s_a u,a\ne i.$ If $a>i$, then  $v=s_a s_0 s_2\cdot s_{i-1}u\notin \Waab$. If $1<a<i$ then we can move $s_a$ to its left until we form a forbidden braid. Finally if $a=1$, then $v=s_0 s_2 s_1 z$; repeating the above argument we see that $z=s_3s_4\cdots s_{n-1}$. But then $\ell(v)=n$, against our assumption.

{\bf Case 3: type $\widehat E_n$ ($n=6,7,8$).} In these cases $h_\s=n-3$.
In type $\widehat E_6$, $\Waab_4$ has Hasse diagram
$$
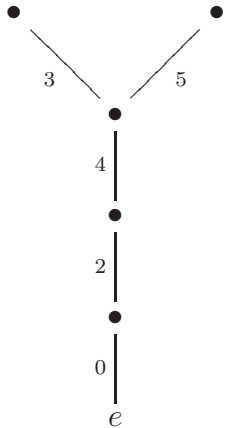
\xymatrix@R+2pt@C+1pt{ 
{\bullet}\ar@{-}[dr]_{3}&&{\bullet}\ar@{-}[dl]^{5}\\
&{\bullet}\ar@{-}[d]_4\\ 
&{\bullet}\ar@{-}[d]_{2}\\
&{\bullet}\ar@{-}[d]_{0}\\
&{e}  
}
$$
so, if $\s_{|\Waab_4}\ne Id$, then $\s_{|\Waab_4}= \s'_{|\Waab_4}$ with $\s'\in Aut(\Pi)$ exchanging $\a_3$ and $\a_5$.

In the other cases, we see that the Hasse diagram of $\Waab_{n-1}$ is
$$
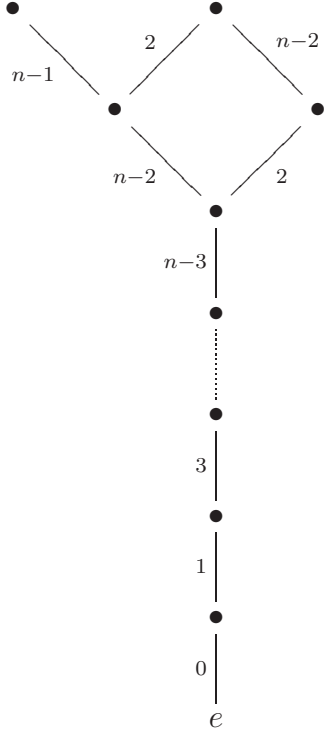
\xymatrix@R+2pt@C+1pt{ 
{\bullet}\ar@{-}[dr]_{n-1}&&{\bullet}\ar@{-}[dl]_{2}\ar@{-}[dr]^{n-2}\\ 
&{\bullet}\ar@{-}[dr]_{n-2}&&{\bullet}\ar@{-}[dl]^{2}\\
&&{\bullet}\ar@{-}[d]_{n-3}\\
&&{\bullet}\ar@{.}[d]\\
&&{\bullet}\ar@{-}[d]_3\\ 
&&{\bullet}\ar@{-}[d]_1\\ 
&&{\bullet}\ar@{-}[d]_{0}\\
&&{e}  
}
$$
From this graph we see that $\s_{|\Waab_{n-2}}=Id$. Thus, in this cases, $Aut(\Ab)=\{Id\}$ and there is nothing to prove.
\end{proof}

\begin{rem}
Note that the proof of Lemma \ref{hsigma} shows also that, if $\s\ne Id$, then $h_\s$ does not depend on $\s$.
\end{rem}

We are now ready to prove our main Theorem:

\begin{proof}[Proof of Theorem \ref{T}] We will prove by induction on $h$ that, given $\s\in Aut(\Ab)$,  $\s_{|\Waab_h}\in Aut(\Pi)_{|\Waab_h}$ for any $h\ge0$. If $h=0$, there is nothing to prove. Assume $h>0$. Then, by the induction hypothesis,  there is $\tilde\s\in Aut(\Pi)$ such that $\s_{|{\Waab_{h-1}}}=\tilde\s_{|{\Waab_{h-1}}}$. Set $\tau=\s\tilde\s^{-1}$. By Lemma \ref{hsigma}, there  is $\s'\in Aut(\Pi)$ such that $\tau_{|\Waab_{h_{\tau}+1}}=\s'_{|\Waab_{h_{\tau}+1}}$. Clearly $h_\tau\ge h-1$, so $\tau_{|\Waab_{h}}=\s'_{|\Waab_{h}}$, i.e. $\s\tilde\s^{-1}_{|\Waab_{h}}=\s'_{|\Waab_{h}}$, so that $\s_{|{\Waab_{h}}}=(\s'\tilde\s)_{|{\Waab_{h}}}$. 
\end{proof}

\section{Symmetries of the Hasse graph of $\Ab$}
Recall that  $H_\Ab$  is the Hasse diagram of $\Ab$. We identify $\Ab$ with either $\Wab$ or the set of alcoves $C_\i,\,\i\in\Ab$ (cf \eqref{alc}).
\begin{lemma}\label{L} If $f\in Aut(H_\Ab)$ is such that $f(e)=e$, then $f\in Aut(\Ab)$. \end{lemma}
\begin{proof} It suffices to prove that for $w\in\Wab$, we have $\ell(w)=\ell(f(w))$.
In fact, if $v,w\in\Wab,v<w,$ there exists $v=v_0<v_1<\dots<v_k=w,\,v_i\in\Wab$ with $\ell(v_i)=\ell(v)+i$, hence we need to prove just that $f(v_i)<f(v_{i+1})$. Since $f(v_i)$ has to be linked in $H_\Ab$  to
$f(v_{i+1})$, the fact that $\ell(f(v_{i+1}))=\ell(f(v_i))+1$ implies $f(v_i)<f(v_{i+1})$.\par
We perform an  induction on $\ell=\ell(w)$. The claim  is true by assumption if $\ell=0$ and follows from Remark \ref{51}  if $\ell=1$.
Now, if $w\in\Wab, \ell(w)=k, k>1$, then  $w$ is linked to $v\in\Wab$,  with  $\ell(v)=k-1$. Then $\ell(f(v))=k-1$, hence  either  $\ell(f(w))=k-2$ or $\ell(f(w))=k$; the first case can't occur, since 
by induction $f$ permutes the elements of length $k-2$, hence $\ell(f(w))=k$, as required.
\end{proof}

\begin{lemma}\label{edges}
The number of edges connected to a node $w$ in $H_\Ab$ is equal to the number of roots in $\a \in\Pia$ such that $w(\a)\in \pm(\d-\Dp)$.
\end{lemma}
\begin{proof}
If $v$ and $w$ are connected by an edge then $vs_i=w$ with $\ell(w)=\ell(v)\pm1$. If $\ell(w)=\ell(v)+1$ then $v(\a_i)\in N(w)$, hence $v(\a_i)\in\d-\Dp$, so $w(\a_i)=-v(\a_i)\in-(\d-\Dp)$. If $\ell(w)=\ell(v)-1$, then $ws_i=v$ hence $w(\a_i)\in N(v)$, so $w(\a_i)\in\d-\Dp$.
\end{proof}

\begin{prop}\cite{S} ÊIf $\g$ is not of type $C_3, G_2$, then $Aut(H_\Ab)=Aut(\Pia)$. If $\g$ is  of type $C_3, G_2$,  $Aut(H_\Ab)=Aut(\Pia)\times \ganz/2\ganz$.
\end{prop}
\begin{proof} Let  $f\in Aut(H_\Ab)$. Let $w=f(e)$.  The set of faces of $w(C_1)$ is given by the hyperplanes corresponding to the roots in $w(\widehat\Pi)$.  By Remark \ref{51} only one edge is connected to $w$. Thus, by Lemma \ref{edges},  $w(\widehat\Pi)$ contains exactly one root in $\pm(\d-\Dp)$. Let $\a_j\in\Pia$ be the simple root such that $w(\a_j)\in\pm(\d-\Dp)$. All other walls of $w(C_1)$ are walls of $2C_1$. Thus there is $\be\in \Pi\cup\{\a_0+\d\}$ such that  the hyperplanes corresponding to $(\Pi\cup\{\a_0+\d\})\backslash \{\be\}$ are walls of $w(C_1)$. It follows that there is a vertex $2o_i$ (the intersection of all the hyperplanes in $(\Pi\cup\{\a_0+\d\})\backslash \{\be\}$) that is in the closure of $w(C_1)$ .  It suffices to prove that, if $i\ne 0$, then $m_i=1$.
Indeed, if this is the case, there exists $z\in Z_2$ such that $z\,f(C_1)=C_1$, hence we may apply Lemma \ref{L} and Theorem \ref{T}.

We now prove that $m_i$ is odd. Let $c_i(\a)$ denote the coefficient of $\a_i$ in the expansion of $\a$ in terms of simple roots. If $m_i$ is even, then there is a root $\a\in\Dp$ such that $c_i(\a)=m_i/2$. Then $(\d-\a)(2o_i)=0$, so the hyperplane corresponding to $\d-\a$ passes through $2o_i$ and meets the interior of $2C_1$ (see \cite{CP}). Thus the hyperplanes corresponding to $(\Pi\cup\{\a_0+\d\})\backslash \{\be\}$ cannot be all walls of $w(C_1)$.

This argument already finishes the proof in all classical cases, for, in these cases, we have that $m_i\le 2$.


It remains to deal with the exceptional cases. We first prove that, letting $\i\in\Ab$ be the ideal corresponding to $w$, then 
$$
\Phi_\i=\{\a\in\Dp\mid c_i(\a)>\frac{m_i}{2}\}.
$$
Since we are assuming that $i\ne 0$, $\i$ is maximal in $\Ab$. Since $2o_i\in \ov{w(C_1)}$, we have $w^{-1}(2o_i)\in \ov{C_1}$. If $\a\in\i$, then $w^{-1}(\d-\a)\in-\Dap$, hence $w^{-1}(\d-\a)(w^{-1}(2o_i))=(\d-\a)(2o_i)\le 0$. It follows that $1-\frac{2c_i(\a)}{m_i}\le 0$, or equivalently $c_i(\a)\ge \frac{m_i}{2}$. Since $m_i$ is odd, we see that 
$$
\Phi_\i\subset\{\a\in\Dp\mid c_i(\a)>\frac{m_i}{2}\}.
$$
Since $\{\a\in\Dp\mid c_i(\a)>\frac{m_i}{2}\}$ is clearly an abelian dual order ideal in $\Dp$ hence, since $\i$ is maximal, equality holds.

Our argument reduces the missing cases to a few direct inspections: the graph $H_\Ab$ near $1$ is of this type:
$$
\xymatrix@R+2pt@C+1pt{ 
{\bullet}\ar@{-}[d]_{i_h}\\ 
{\bullet}\ar@{.}[d]\\ 
{\bullet}\ar@{-}[d]_{i_1}\\
{e}  
}
$$
with $h=4$ in type $F_4$, $h=3,4,6$ in type $E_6$, $E_7$, $E_8$ respectively. Thus, the graph near $\i$, has to be a chain of the same length.
Using the explicit description of $\i$ given above, it is easy to determine the structure of the subposet $\{\mathfrak{j}\in\Ab\mid \j\subset\i\}$ and verify that, if $m_i>1$, then its Hasse graph near the maximum is a chain of  length strictly less than $h$.
\end{proof}

\providecommand{\bysame}{\leavevmode\hbox to3em{\hrulefill}\thinspace}
\providecommand{\href}[2]{#2}

\vskip5pt
\footnotesize{

\noindent{\bf P.C.}: Dipartimento di  Ingegneria e Geologia,
 Universit\`a  di Chieti-Pescara, Viale Pindaro 42,  65127 Pescara, Italy; 
{\tt cellini@sci.unich.it }

\noindent{\bf P.M.F.}: Politecnico di Milano, Polo regionale di Como, 
Via Valleggio 11, 22100 Como,
Italy; {\tt pierluigi.moseneder@polimi.it}

\noindent{\bf P.P.}: Dipartimento di Matematica, Sapienza Universit\`a di Roma, P.le A. Moro 2,
00185, Roma, Italy; {\tt papi@mat.uniroma1.it}

\end{document}